\documentclass[11pt]{article} 

\usepackage[utf8]{inputenc} 

\usepackage{amsmath,amssymb,amsfonts,amsthm}
\usepackage{geometry} 
\usepackage{graphicx} 

\usepackage{booktabs} 
\usepackage[dvipsnames]{xcolor}
\usepackage[disable]{todonotes}
\usepackage{array} 
\usepackage{paralist} 
\usepackage{verbatim} 
\usepackage{subfig} 
\usepackage{ytableau}

\usepackage{fancyhdr} 
\usepackage{sectsty}
\allsectionsfont{\sffamily\mdseries\upshape} 

\usepackage[nottoc,notlof,notlot]{tocbibind} 
\usepackage[titles,subfigure]{tocloft} 

\newtheorem{conjecture}{Conjecture}

\title{Kostka Numbers and Longest Increasing Subsequences}
\author{Arjun Krishnan \and Scott Neville}
\date{}

\newcommand{\Secref}[1]{Section~\ref{#1}}
\newcommand{\Thmref}[1]{Theorem~\ref{#1}}

\newcommand{\Lemref}[1]{Lemma~\ref{#1}}

\newcommand{\AND}{\textrm{ and }} 
\newcommand{\OR}{\textrm{ or }}

\renewcommand{\emph}[1]{\textbf{#1}}

\theoremstyle{plain}
\newtheorem{theorem}{Theorem}[section]
\newtheorem*{theorem*}{Theorem}
\newtheorem{lemma}[theorem]{Lemma}
\newtheorem*{lemma*}{Lemma}
\newtheorem{corollary}[theorem]{Corollary}

\newtheorem*{prop*}{Proposition}

\theoremstyle{definition}
\newtheorem{define}[theorem]{Definition}
\newtheorem{algorithm}[theorem]{Algorithm}

\newtheorem{example}{Example}
\newtheorem*{example*}{Example}
\newtheorem{claim}[theorem]{Claim}

\theoremstyle{remark}

\newtheorem*{remark*}{Remark}

\usepackage[sorting=none,natbib=true,citestyle=authoryear,style=numeric,backend=biber]{biblatex}
\newcommand{\sqtableau}[2]{\square_{#1 \times #2}}
\newcommand{\SNminus}{\ensuremath{-}}
\newcommand{\SNcomplement}[2][]{\overline{#2}^{#1}}
\newcommand{\kostka}[2]{K_{#1 #2}}
\newcommand{\kmu}[2][]{\mu_{#2}^{#1}} 
\newcommand{\colcomp}[2][n]{\overline{#2}^{#1}}

\newcommand{\lislessk}[2]{\mathcal{L}_{#1}(#2)}

\begin{document}
\maketitle

\begin{abstract}
A classical bijection relates certain Kostka numbers, the Catalan numbers, and permutations of length $n$ with longest increasing subsequence (LIS) of length at most $2.$ We generalize this bijection and find Kostka numbers which count the number of permutations of $n$ with LIS length at most $w,$ the number of permutations with $(1, \cdots, w)$ as a LIS, and other similar subsets of permutations.
\end{abstract}


\section{Introduction}

The Kostka numbers $\kostka{\lambda}{\mu}$ appear naturally in symmetric function theory and in several counting problems related to the symmetric group $S_n.$ They are indexed by a partition $\lambda$ of $n$ (written $\lambda \vdash n$) and a vector with non-negative integer entries $\mu$ such that $\sum_i \mu_i = n$. They count the number of Young tableaux with shape $\lambda$ that are filled with $\mu_i$ copies of the number $i$. $\mu$ is called the \emph{content} (or \emph{weight}) vector. 

The Kostka numbers $K_{\lambda \mu}$ were originally \citep{kostka1882ueber} defined as the monomial coefficients of the Schur functions $(s_\lambda)$
$$s_\lambda = \sum_{\mu \, \vdash n} K_{\lambda \mu} m_\mu$$
where $\lambda \vdash n$, and $m_\mu$ is the monomial symmetric function. \citet{littlewood1938construction} later discovered the combinatorial interpretation we gave above. 

Given $\sigma = (\sigma_1, \sigma_2, \cdots, \sigma_n) \in S_n$, $(\sigma_{i_1}, \sigma_{i_2}, \ldots , \sigma_{i_w})$ is an increasing subsequence in $\sigma$ if  $i_1 < i_2 < \cdots < i_w$  and $\sigma_{i_1} < \sigma_{i_2} < \cdots < \sigma_{i_w}$. Consider the set $\lislessk{n}{w} \subset S_n$, whose elements have longest increasing subsequences (LIS) of length at most $w$. Let $u_n(w) = | \lislessk{n}{w} |$. It is a classical fact that $u_n(2) = C_n$, where $C_n$ is the $n$\textsuperscript{th} Catalan number. $u_n(2)$ is related to a certain Kostka number through a classical operation on Young tableaux (discussed below). We generalize this bijection to count $\lislessk{n}{w}$ and other subsets such as $\{ \sigma \in S_n \colon (1,2,\ldots w) \text{ is an LIS} \}$ for $w > 2$ using related Kostka numbers. Computational evidence suggests that the latter Kostka numbers also count $r$-colored non-crossing partitions of $[n]$ \citep{MR3139391}, which are themselves counted by certain vascillating tableaux \citep{MR2272140,MR2873884} (see Conjecture \ref{conj:len k end k and noncrossing}).

Gessel provided a formula for $u_n(w)$ in terms of determinants of modified Bessel functions \citep{MR1041448}; these formulas simplify for $w \leq 3$ and may be found in the same paper. By the Robinson-Schensted (RS) correspondence, $u_n(w)$ are counted by pairs of standard tableaux of the same shape with at most $w$ columns. Let $y_n(w)$ be the number of standard tableaux with at most $w$ columns. Regev gives formulas for $y_n(w)$ for $w = 2 \AND 3$, and \citet{MR977181} gives combinatorial proofs and formulas for $2 \leq w \leq 5$. Wilf \cite{MR1156657} gives a formula for $u_n(w)$ for $w = 2,4,\ldots$ in terms of $y_j(w)$, $j=1,\ldots,2n$; thus formulas for $u_n(w)$ for $w=4$ may be obtained from the formulas for $y_n(w)$. \citet{MR1080995} gives formulas for both $y_n(w)$ and $u_n(w)$ for large $w$ (depending on $n$). Many of these facts and references may be found in Stanley's book \citep[pp 452,493,problem 7.16]{MR1676282}. 


The Robinson-Schensted (RS) correspondence is a bijective map between permutations in $S_n$ and pairs of standard Young tableau of the same shape $\lambda \vdash n$ \citep{MR0062127,MR0121305}. The length of the top row (or width) of $\lambda$ is the length of a LIS. Thus, via the RS correspondence, $\lislessk{n}{2}$ is in bijection with pairs of tableaux $(P,Q)$ with width at most $2$ and size $n$. There is an appealing, classical way to convert this pair of tableaux into a single rectangular tableau of shape $\sqtableau{2}{n}$, where $\sqtableau{w}{n}$ denotes a rectangular diagram with width $w$ and height $n$ \citep{macmahon1915combinatorial}; \citep[page 263, problem 6.19.ww]{MR1676282}. First replace $i$ with $2n-i+1$ for $i \in [n]$ in the tableau $Q$. Then, rotate $Q$ by $180$ degrees and glue it to $P$; the gluing is done so that the rows with length $1$ in $P$ align with the rows of length $1$ in $Q$. Thus we obtain a rectangular tableau with width $2$ and height $n$. We illustrate this cutting, relabeling and gluing procedure in Example \ref{ex:basic catalan number cutting and gluing procedure}.

\begin{example}
    \label{ex:basic catalan number cutting and gluing procedure}
$$
\begin{ytableau}
1 & 2 \\
3 
\end{ytableau} \quad
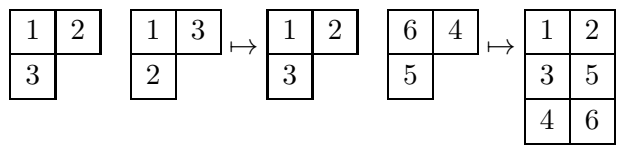
\begin{ytableau}
1 & 3 \\
2
\end{ytableau} \mapsto
\begin{ytableau}
1 & 2 \\
3 
\end{ytableau} \quad
\begin{ytableau}
6 & 4 \\
5
\end{ytableau} \mapsto
\begin{ytableau}
1 & 2 \\
3 & 5 \\
4 & 6
\end{ytableau}  
$$
\label{example:simple rectangular transformation for catalan numbers and lis of length at most 2}
\end{example}

It is not hard to see that the steps of the transformation are invertible, and that it is a bijection. Let $\vec{a}_k$ be a vector of $k$ copies of the number $a$. The number of rectangular tableaux of shape $2 \times n$ is given by the Kostka number $\kostka{\sqtableau{2}{n}}{\vec{1}_{2n}}$. Since the rectangular tableaux are standard ---all the weights are $1$--- we can use the hook length formula \citep{MR0062127} to count them. Hence,
\begin{equation}
    u_n(2) = \kostka{\sqtableau{2}{n}}{\vec{1}_{2n}} = \frac{(2n!)}{n! (n+1)!} = C_n.
    \label{eq:kostka catalan numbers}
\end{equation}

\section{Main results}

We first state a generalization of \eqref{eq:kostka catalan numbers}. 
\begin{theorem}
    Let $\kmu[n]{a} = \vec{a}_n \oplus \vec{1}_n$ be the vector of $n$ copies of $a$ and $n$ copies of $1$. Then,
    \[
        \kostka{\sqtableau{w}{n}}{\kmu[n]{w-1}} =  u_n(w). 
    \]
    \label{thm:kostka numbers same as number of lis of length at most k}
\end{theorem}

The Kostka numbers in \Thmref{thm:kostka numbers same as number of lis of length at most k} count rectangular tableaux with particular weight vectors. The proof establishes a bijection between  these rectangular tableaux and $\lislessk{n}{w}$ by generalizing the argument that proves \eqref{eq:kostka catalan numbers}. The generalization involves a standard involutive operation on columns \citep[page 473, problem 7.41]{MR1676282} (see section \ref{sec:algorithm}).

Next, we consider a further generalization of the algorithm by skewing the weight vector $\mu$.  The \emph{skew} weight vectors have an additional parameter, $k \in \mathbb{Z}$. We define 
\begin{equation}
    \kmu[n,k]{a} =  \vec{a}_{n-k} \oplus \vec{1}_{ak + n}
    \label{eq:definition of skew weights}
\end{equation}
so that $\sum_i \left( \kmu[n,k]{a} \right)_i = (a+1)n$; this last constraint ensures that the weights fill rectangular tableaux of width $a+1$ and height $n$. The magnitude of $k$ indicates how asymmetric the weight vector is. Clearly we must have $k \leq n$, and the constraint $n + k a \geq 0$ implies that $k$ must be an integer in the interval $[-\frac n a, n]$. When the number $n$ is understood, we will simply use $\kmu[k]{a}$. If $k = 0$, we will drop $k$ from the notation and simply write $\kmu{a}$, since $\kmu[n]{a} = \kmu[n,0]{a}$.

We will also use the following generalization of the Catalan numbers $A_{n,m}$ from \citet{Griffiths2011}. These count the number of standard rectangular tableaux of width $n$ and height $m.$
 When $m=2$, again, these are the Catalan numbers. Using the hook-length formula, simple formulas for $A_{n,m}$ may be obtained:
$$A_{n,m} = \frac{(mn)! s(m-1) s(n-1)}{s(m+n-1)}$$
where $s$ denotes the super factorial function $s(k) = \prod_{i=1}^k i!$.  
\begin{theorem}
For all $n$, and any $2 \leq w < n$, fix a nonzero integer $k$ in the interval $[-n/(w-1), n]$. Let
\[
    m =
    \begin{cases}
        n+k(w-1) & k > 0 \\
        n - k & k < 0
    \end{cases}
\]
Then, $K_{\sqtableau{w}{n}, \kmu[n,k]{w-1}} A_{|k|,w}$ counts the number of permutations $\sigma \in S_m$ that have at least $|k|$ disjoint longest increasing subsequences of length $w$ using the numbers from $\{1, \ldots , |wk|\}$.
    \label{thm:generalized theorem with skewed tableau}
\end{theorem}
The $k=0$ case is covered in \Thmref{thm:kostka numbers same as number of lis of length at most k}.
The proof follows the algorithm in \Secref{sec:algorithm} along with some tweaks to work with the skew weights (see \Secref{sec:skewed weights}).
The $k=1$ case gives the following corollary.

\begin{corollary}
 
When $k=1$, the algorithm used to prove \Thmref{thm:generalized theorem with skewed tableau} provides a bijection between rectangular tableaux of shape $\sqtableau{w}{n}$ and weight $\kmu[n,1]{w-1}$, and the set $\{\sigma \in  S_n \colon (1,2, \ldots, w) \text{ is an LIS}\}$. Hence,

\begin{equation}
    \kostka{\sqtableau{w}{n}}{\kmu[n,1]{w-1}} = \left| \{\sigma \in  S_n \colon (1,2, \ldots, w) \text{ is an LIS}\}  \right|.
    \label{eq:specific LIS case}
\end{equation}

\end{corollary}

The following is a particular case of the symmetry noted by \citet[Theorem 5]{MR3386516}. 

\begin{corollary} \label{thm:kostka numbers rectangular symmetry}
$$ \kostka{\sqtableau{w}{n}}{\kmu[n,k]{w-1}} = \kostka{\sqtableau{w}{n+k(w-2)}}{\kmu[n+k(w-2),-k]{w-1}}.$$
\end{corollary}

The authors originally observed a relationship between the RHS of \eqref{eq:specific LIS case} and $k$-colored non-crossing partitions, originally defined by \citet{MR3139391}. 
Let $\operatorname{NC}_2(n,k)$ be the number of noncrossing partitions of $[n]$ with $k$ colors \citep[Corollary $1.5$]{MR3139391}.

\begin{conjecture}
$| \{\sigma \in  S_n \colon (1,2, \ldots, w) \text{ is an LIS}\} | = NC_2(n-w+1, w-1).$ %
\label{conj:len k end k and noncrossing}
\end{conjecture}

\subsection{Acknowledgements}

Computer simulation indicated Conjecture \ref{conj:len k end k and noncrossing}. We are indebted to E.\,Marberg for pointing us towards these particular Kostka numbers, and this inspired our Theorem \ref{thm:generalized theorem with skewed tableau}.  He mentioned hearing the following conjecture from A.\,Tripathi: 

\[
    \kostka{\sqtableau{3}{n}}{\kmu[n,1]{2}} = \operatorname{NC}_2(n,2).
\]

S.\,Neville thanks the REU program at the University of Utah for their support.

\section{Tableau Operations} 
\label{sec:shortdefs}

We represent a Young diagram as $\lambda = (\lambda_1, \cdots, \lambda_m)$ with $\lambda_i$ being the number of boxes in the $i$\textsuperscript{th} \emph{column}. This is \emph{nonstandard}, since $\lambda_i$ usually denotes to the length of the $i$th row, but it makes the upcoming tableau complement operation easier to state.
 Recall that a Young tableau has each column strictly increasing and each row non-decreasing. We call a Young tableau a standard Young tableau if rows are strictly increasing as well.  We use $[n]$ to denote the set $\{1, \cdots, n\}$. If $P$ is a Young tableau, we will use $P_i$ to denote the $i$th column.


We have two operations on tableaux that we need to introduce. One is a standard operation, but is phrased in a nonstandard way since it makes the second easier to state. If $\lambda$ is a sub-diagram of $\sqtableau{w}{h}$ ($\lambda \subset \sqtableau{w}{h}$), we define $\sqtableau{w}{h} - \lambda$ as the sub-diagram of $\sqtableau{w}{h}$ with a $180^{\circ}$ rotated $\lambda$ removed from the bottom right corner of $\sqtableau{w}{h}$. In more standard language, let $\lambda'$ be a skew diagram such that  $\sqtableau{w}{h} / \gamma = \lambda'$, where $\gamma \subset \sqtableau{w}{h}$. Let $\lambda$ be a $180^{\circ}$ rotated version of $\lambda'$. Then, in our notation, $\gamma = \sqtableau{w}{h} \SNminus \lambda$.

\begin{define}[Column]
   A column is a column vector $C = (c_1,\ldots,c_k)$ whose elements are strictly increasing. It will typically represent the column of a tableau. When we refer to a set as a column, we mean that its elements are arranged to be strictly increasing. Let $|C|$ be its cardinality. 
\end{define}

\begin{define}[Rectangular diagram/tableau subtraction]
\label{def: tableau subtraction}
Let $\lambda \subset \sqtableau{w}{h}.$ Then $\gamma = \sqtableau{w}{h}-\lambda$ is a Young diagram with columns $\gamma_i$ that satisfy
\[
    |\gamma_i| = h - |\lambda_{w-i+1}|.
\]
When $R$ is a tableau of shape $\sqtableau{w}{n}$ and $Q$ is a subtableau, we similarly remove a $Q$ shaped block from it, with $(R-Q)_i$ being composed of the $h-|Q_{w-i+1}|$ smallest elements in $R_i$. 
\end{define}

For example 
\[
     \lambda = \ydiagram{3,1} \qquad \sqtableau{3}{4} - \lambda = 
\ydiagram{3,3,2}
\]

\begin{define}[Column complement]
    The complement of a column $C$ with respect to a set $[n]$ consists of the column with elements $[n] \setminus C$, and is denoted $\colcomp[n]{C}$. We omit the superscript $n$ if it is clear from context. The operation also applies to empty columns.
\end{define} 

 Note that 
\[
    \colcomp{\colcomp{C}} = C,
\]
that is, the operation is involutive. If we complement all the columns of a tableau and reverse the column order, we obtain a new tableau. (See \Lemref{lem:complement operation produces a tableau}.)

\begin{define}[Tableau complement]
The complement of a tableau $P$ with respect to a width $w$ and set $[n]$ is denoted by $\SNcomplement[w,n]{P}$, where
\[
    (\SNcomplement[w,n]{P})_i = \colcomp{P_{w-i+1}}  \quad 1 \leq i \leq w.
\]
\end{define}
Again, we occasionally omit the number $n$ when it is clear from context. A column $P_{w-i+1}$ may be empty, and in this case $(\SNcomplement[w,n]{P})_i = [n]$.

\begin{example}
\[
P = 
\begin{ytableau}
1 & 1 & 2 \\
2 & 3 & 4 \\
3 & 4 
\end{ytableau}
\qquad \qquad
\SNcomplement[3,4]{P} = 
\begin{ytableau}
1 & 2 & 4 \\
3 
\end{ytableau} 
\]
\end{example}

\begin{lemma}
If $P$ is a tableau then so is $\SNcomplement[w,n]{P}$ for all $w,n \in \mathbb{N}$.
\label{lem:complement operation produces a tableau}
\end{lemma}

Lemma \ref{lem:complement operation produces a tableau} follows from Claim \ref{claim:column complement correctly reverses order} in \Secref{sec:algorithm}. 

\begin{lemma}
The map $P \mapsto \SNcomplement[w,n]{P}$ is invertible, and is its own inverse.
\label{lem:complement operation is a bijection}
\end{lemma}

\begin{proof}
    If $M = \SNcomplement[w,n]{P}$, the complement of the $w - i + 1$\textsuperscript{th} column of $M$ is equal to the $i$\textsuperscript{th} column of $P$. That is, for $1 \leq i \leq w,$
    \begin{align*}
        \SNcomplement[n]{M_{w-i+1}} 
        & = ([n] \SNminus M_{w - i + 1}) = ([n] \SNminus ([n] \SNminus P_{w-(w - i + 1)+1})) \\
        & =([n] \SNminus ([n] \SNminus P_i) = P_i.
    \end{align*}
\end{proof}

\section{Algorithm}
\label{sec:algorithm}

We now provide an invertible algorithm which takes a $\sqtableau{w}{n}$ shaped tableau with weights $\kmu[n]{w-1}$, and returns a pair of tableaux $(P,Q)$ with shape $\lambda$ and width at most $w$. The triple $(P, Q, \lambda)$ in turn corresponds to a permutation with LIS of length at most $w$ by the RS correspondence.

Let $R$ be a semistandard tableau of shape $\sqtableau{w}{n}$ with weight $\kmu[n]{w-1}$. 
\begin{enumerate}
    \item The elements $[n+1, 2n]$ in $R$ must form a contiguous skew tableau in the bottom right corner of $R$. If we rotate this skew tableau by $180 \deg$ and replace the numbers $k$ with $2n-k + 1$ for $k \in [n+1,2n]$, we obtain a standard tableau $Q$ of shape $\lambda$. 
        For all coordinates $(i,j)$ in the diagram $\lambda$ we have $Q_{i,j} = 2n - R_{w-i+1, n-j+1} + 1 $.
    \item To get $P$, we remove the elements $[n+1, 2n]$ from $R$, and then take its tableau complement: 
        \[
            P = \SNcomplement[w,n]{R - Q}
        \]
\end{enumerate}

\begin{example}
If $$ R = 
\begin{ytableau}
1 & 1 & 2 \\
2 & 3 & 4 \\
3 & 4 & 5 \\
6 & 7 & 8
\end{ytableau}
$$
First, we split the $R$ into $R \SNminus \lambda$ and a piece that will eventually become $Q$. 
$$ \begin{ytableau}
1 & 1 & 1 \\
2 & 3 & 4 \\
3 & 4 
 \end{ytableau} \quad
 \begin{ytableau}
 \none & \none  & 5 \\
 6 & 7 & 8
  \end{ytableau}$$
  
Then, we take the complement of the first tableau to get $P$. To get $Q$, we rotate and replace the numbers as specified above. The algorithm returns:
$$P = 
\begin{ytableau}
1 & 2 & 4 \\
3 
\end{ytableau} \quad
Q = 
\begin{ytableau}
1 & 2 & 3 \\
4 
\end{ytableau} 
 $$
\end{example}

\begin{lemma}
    The $P$ given by the algorithm is a standard tableau, and has the same shape $\lambda$ as $Q$.
    \label{lem:complement produces a standard tableau}
\end{lemma}
\begin{proof}
    The proof follows from Claims \ref{claim:complement gives the right shape}, \ref{claim:complement gives correct entries in tableau} and \ref{claim:column complement correctly reverses order} below. They prove that the complement operation on $R \SNminus Q$ gives the correct shape $\lambda$ for $P$ (by \ref{claim:complement gives the right shape}), $P$ has the right entries and that $P$ is a standard tableau (by Claim \ref{claim:complement gives correct entries in tableau} when $k=0$ and Claim \ref{claim:column complement correctly reverses order}).
\end{proof}

\begin{lemma}
    The algorithm is a bijection from rectangular tableau of shape $\sqtableau{w}{n}$ and weight $\kmu{w-1}$ to pairs of standard tableau of the same shape having $n$ elements and width at most $w$.
\end{lemma}
\begin{proof}
    The algorithm is reversible at every stage: the splitting of $R$ into $R \SNminus \lambda$ and $Q$ is invertible; the relabeling of the elements in $Q$ is a bijection from $\{1,\ldots,n\}$ to $\{n+1,\ldots,2n\}$; and the complement operation is an invertible operation (\Lemref{lem:complement operation is a bijection}) that sends a tableau of shape $\sqtableau{w}{n} \SNminus \lambda$ and weight $\kmu{w-1}$ to a standard tableau of shape $\lambda.$ We may invert the algorithm by reversing the steps: given a pair of tableaux $(P,Q)$ of shape $\lambda \vdash n$ and width at most $w$, we can reconstruct a $\sqtableau{w}{n}$ shaped tableau $R$ with weight $\kmu{w-1}$. Since each of the stages of the algorithm are bijections, the algorithm itself is a bijection.
    
\end{proof}

\begin{claim}[Complement gives the right shape]
    Let $\lambda$ be a Young diagram. Let $A$ be a tableau with shape $\sqtableau{w}{n} - \lambda.$ Then, $\SNcomplement[w,n]{A}$ has shape $\lambda$.
    \label{claim:complement gives the right shape}
\end{claim}

\begin{proof}
    By Definition \ref{def: tableau subtraction} there are exactly $n - |\lambda_{w - i + 1}|$ distinct elements in column $A_i$, since $A$ is formed by subtracting a \emph{rotated} $\lambda$ from the square tableau of height $n$. The elements are distinct because columns in a tableau must be strictly increasing. Therefore, 
    \[
        |(\SNcomplement[w,n]{A})_{w-i+1}| = n - |A_i| = \lambda_{w - i + 1}.
    \] 
    This shows that all the columns of $\SNcomplement[w,n]{A}$ have the correct height.
\end{proof}

The following claim is stated in a general form so that it applies to the skew tableau considered in \Secref{sec:skewed weights}. We only need the case $k = 0$ for \Lemref{lem:complement produces a standard tableau}.

\begin{claim}[Complement gives the right entries]

For integers $n,w > 1$, and $k$ such that $ -n/(w-1) \leq k \leq n$, let
$\lambda \vdash n+k(w-1)$ be a diagram with width at most $w$ and let $A$ be a tableau with shape $\sqtableau{w}{n} - \lambda$ 
with weight $\overrightarrow{(w-1)}_{n-k}$. Then, $\SNcomplement[w,n-k]{A}$ contains exactly one copy of 
each element in $[n-k]$.
\label{claim:complement gives correct entries in tableau}
\end{claim}

\begin{proof}
Since there are at most $w$ columns in $A$, and no column can have a duplicate element (if
a column did then it would not be strictly increasing), there can be at most one
column that does not contain some element $v$.  Let $\rho(v)$ be the only column
of $A$ such that $A_{\rho(v)}$ does not contain $v \in [n-k]$.  This
means that $v \in (\SNcomplement[w,n-k]{A})_{w - \rho(v) +1}$, and so all elements $[n-k]$
appear in $\SNcomplement[w,n-k]{A}$. Since $\rho(v)$ is the unique column that does not contain each $v \in [n-k]$, this
also shows that it appears exactly once in $\SNcomplement[w,n-k]{A}$; since the complement is with respect to $n-k$, these are the only elements that appear in $\SNcomplement[w,n-k]{A}$.
\end{proof}

For $2$ columns $w, v$, we say $w \preceq v$ if $|w| \geq |v|$ and $\forall i \leq |v|, w_i \leq v_i$. Thus $P$ is a semistandard tableau with columns $P_i$ if for all $i \leq j$ we have $P_i \preceq P_j$.

\begin{claim}[Column complement reverses $\preceq$ order]
    Let $w \preceq v$ be two columns. Then, $\colcomp[]{v} \preceq \colcomp[]{w}$. 
    \label{claim:column complement correctly reverses order}
\end{claim}

\begin{proof}
Let $U = [n]$ be a set of integers such that $v,w \subset U$. We omit the superscript in $\colcomp[n]{x}$ in the following.
Let $w$ and $v$ be two columns such that $w \preceq v$. The proof proceeds by induction on the elements of $U$. We imagine that we are growing $\overline{w}$ and $\overline{v}$ by examining each element of $U$ sequentially and adding it to either $\left( \overline{v} \right)_q \OR \left( \overline{w} \right)_q$, where $\left( \overline{v} \right)_q := \{ x \in U \setminus v, x \leq a_q \}$ to be the elements of $U$ added to $\left( \overline{v} \right)_q$ by step $q$ and $\left( \overline{w} \right)_q$ similarly. 

Let $\left( \overline{v} \right)_0$ and $\left( \overline{w} \right)_0$ be empty, and suppose the first $q$ numbers have been examined. As the induction hypothesis, suppose $|( \overline{v} )_q| \geq |(
\overline{w})_q|$ and $\overline{v}_i \leq \overline{w}_i$ for $i=1,\ldots,|(
\overline{w})_q|$. Then, there are four possibilities for the $(q+1)$\textsuperscript{th} element: 
\begin{enumerate}
    \item \label{case1} $q+1 \in v \cap w$, then $q+1$ is not added to either $\left( \overline{v} \right)_q \OR \left( \overline{w} \right)_q$. 
    \item \label{case2} $q+1 \in \overline{v} \cap w,$ then $q+1$ is added to $\left( \overline{v} \right)_q$.
    \item \label{case3} $q+1 \in v \cap \overline{w}$, then $q+1$ is added to $\left( \overline{w} \right)_q$.
    \item \label{case4} $q+1 \in \overline{v} \cap \overline{w}$, then $q+1$ is added to both $\left( \overline{v} \right)_q \AND \left( \overline{w} \right)_q$.
\end{enumerate}
Then, in case \ref{case1}, $q+1$ is not added to either
set, and hence $|( \overline{v} )_{q+1}| = |( \overline{v} )_{q}| \geq |(
\overline{w} )_{q+1}|$ and the induction hypothesis holds.   
In case \ref{case2}, $|( \overline{v} )_{q+1}| \geq |(
\overline{w} )_{q+1}|$ and since no element was added to $(
\overline{w})_{q+1}$, $\overline{v}_{|( \overline{w} )_{q} + 1|} \leq
\overline{w}_{|( \overline{w})_{q} + 1|}$ (if $\left( \overline{w} \right)_{q+1}$ is added eventually).

If we are in case \ref{case3} then we claim that $|( \overline{v})_q| > |( \overline{w} )_q|.$ Suppose for contradiction that 
$\left|( \overline{v} )_q \right| = \left| ( \overline{w} )_q \right|$, 
then for $r = q - |(\overline{w})_q|$ we have $q+1 = v_{r + 1}$, and hence $w_{r + 1} > v_{r+1}$ or $|v| > |w|.$ This contradicts $w \preceq v$. 
Case \ref{case4} is trivial, and hence the induction step is proved. To establish the $q=0$ case, we essentially repeat the above argument with empty $( \overline{v})_0 \AND ( \overline{w})_0$. It is easy to see here that case $3$ cannot occur. The claim is proved when $q = n,$ as then $\left( \overline{v} \right)_q = \overline{v}$ and likewise for $w.$ 
\end{proof}

\section{Skewed Weights}
\label{sec:skewed weights} 

\begin{algorithm}
Generalized Algorithm
\label{algorithm}
\end{algorithm}

Let $2 \leq w \leq n.$ 
\begin{enumerate}
    \item \label{item:form the q'' tableau} Let $R$ be a tableau of shape $\sqtableau{w}{n}$ and content $\kmu[n,k]{w-1}$. It contains $w-1$ copies of the numbers $1$ through $n-k$ and $1$ copy of the numbers $n-k+1$ through $2n + k(w - 2)$. As before, define the skew tableau $Q''$ by choosing cells in $R$ containing the numbers $\{  n-k+1 , n-k+2, \ldots , 2n+k(w-2) \}$.
        Let $Q'$ be the standard tableau obtained by first rotating $Q''$ by 180 degrees and then applying the map $x \mapsto 2n + k(w-2) - x + 1$ to each entry in $Q''$. Let $\lambda$ be the shape of $Q'$.
    \item Let $P'' = \SNcomplement[w, n-k]{R \SNminus Q'}$.
 
        \emph{The following three steps are new.}
    \item \label{item:swapping step} If $k > 0$, let $(P',Q) = (P'',Q')$; otherwise, let $(P',Q) = (Q',P'')$. That is, if $k < 0$ we swap the two tableaux before proceeding. 

    \item \label{item:final renumbering of P'} We apply the map $x \mapsto x + |kw|$ to each element in $P'$ so that it is a standard tableau with values starting at $|kw| + 1$. 
    
    \item 
Let $M$ be a fixed tableau of shape $\sqtableau{|k|}{w}$ containing the elements $[|k|w].$ Let $P$ be the standard tableau obtained by placing $M$ on top of $P'$ such that the first columns of $P'$ and $M$ are aligned.

\end{enumerate}

Let $P' \AND Q$ have shapes $\lambda' \AND \lambda$ respectively. In Lemma \ref{lem:forming P from P' and the rectangular head is a proper tableau}, we show that each column of $\lambda'$ has $|k|$ fewer entries than the corresponding column in $\lambda$, and step $5$ produces a standard tableau of shape $\lambda$. There are $A_{|k|,w}$ choices for the rectangular standard tableau $M$ and hence it appears in \Thmref{thm:generalized theorem with skewed tableau}.
 
\begin{lemma}
    Placing a standard tableau $M$ of shape $\sqtableau{|k|}{w}$ on top of $P'$ such that the first columns of $P'$ and $M$ are aligned, produces a standard tableau $P$ of the same shape as $Q$.
    \label{lem:forming P from P' and the rectangular head is a proper tableau}
\end{lemma}

\begin{proof}[Proof of \Lemref{lem:forming P from P' and the rectangular head is a proper tableau}]
Consider $P''$ and $Q'$ before the swapping in step \ref{item:swapping step}. Recall that $P''_i = \SNcomplement[n-k] {(R \SNminus Q')_{w-i+1}}$ from the algorithm. Also recall that $|Q'_i| + |(R \SNminus Q')_{w-i+1}| = n$ since $R$ has $n$ rows. {Then, from step $2:$}

$$|P''_i| = |\SNcomplement[w,n-k]{(R \SNminus Q')_{w-i+1}}| = n-k - |(R \SNminus Q')_{w-i+1}|$$
$$ = n - k - (n-|Q_i|) = |Q_i| - k.$$
This works when both $k < 0$, and $k > 0$, and implies that either $P''$ has $-k$ more rows than $Q$, or that $P''$ has $k$ fewer rows than $Q$. This holds for all the columns $1 \leq i \leq w$ and all $k$ in the range specified in \Thmref{thm:generalized theorem with skewed tableau}. So this implies that if we attach a square tableau of shape $\sqtableau{|k|}{w}$ to the top of $P'',$ then the columns of $P''$ and $Q'$ have the same sizes. Since step $4$ ensures that $P'$ does not contain the numbers in $[|k|w],$ $P$ is a standard tableau.

\end{proof}

\begin{proof}[Proof of \Thmref{thm:generalized theorem with skewed tableau}]
Let $P$ be the tableau obtained by attaching $M$ to $P'$ in step $5.$ {By Lemma \ref{lem:forming P from P' and the rectangular head is a proper tableau}, } we have a pair $(P,Q)$ with the same shape $\lambda$. By the RSK algorithm, this corresponds to a permutation $\sigma$ of size $m = |\lambda|,$ where $m$ is defined in the statement of the theorem.

Since $M$ contains the numbers $1$ through $|k|w$, it is a well known fact that there must be $k$ disjoint increasing subsequences in $\sigma$ each of length $|k|$ using the numbers $1$ through $|k|w$ \citep[Chapter 3, Lemma 1]{MR1464693}. Therefore $P$ must also have at least $|k|$ such disjoint increasing subsequences. 

As before, it is easy to see that all the steps of the algorithm are invertible. The parameters $k,w \AND n$ are fixed from earlier. The reverse algorithm starts with a permutation in $S_m$ with $|k|$ disjoint LIS of length $w$ made up of the numbers $[|k|w]$. We use the RSK algorithm to form the pairs of tableau $(P,Q)$, and drop the first $|k|$ rows of $P$ and relabel its elements to form $P'$. Then, depending on the sign of $k$, we swap the tableau $P'$ and $Q$ to obtain $P''$ and $Q'$. 

We then apply the (invertible) complement operation to $P''$ and then relabel the elements of $Q$ by inverting the map in step \ref{item:form the q'' tableau} to form $Q'$. Then $\SNcomplement[w, n-k]{P''}$ can be joined to a rotated $Q'$ to obtain $R$, a rectangular tableau of shape $\sqtableau{w}{n}$ and content $\kmu[n,k]{w-1}$. 
\end{proof}

\begin{proof}[Proof of Corollary 2.3]
Note that when $k=1$ then the block $M$ has only one choice:
it must consist of the elements $[w]$ in a diagram of shape $\sqtableau{1}{n},$
and we immediately get a bijection between permutations where $(1, 2, \cdots, w)$ is an LIS. 
\end{proof}

\begin{proof}[Proof of Corollary \ref{thm:kostka numbers rectangular symmetry}]

Without loss of generality we may assume $0 < k \leq n,$ (for $k<0,$ replacing $k$ with $-k$ permutes the left and right hand sides of the equation) and $w > 1.$ Fix a $\sqtableau{w}{k}$ shaped tableau $M.$ 

\Thmref{thm:generalized theorem with skewed tableau} shows that $\kostka{\sqtableau{w}{n}}{\kmu[n,k]{w-1}}$ counts the number of pairs $P,Q$ of Young tableaux having the same shape with width $w,$ $M$ occupying the first $k$ rows of $P,$ and $|P| = m = n +k(w-1).$ Analogously, it shows that $\kostka{\sqtableau{w}{n+k(w-2)}}{\kmu[n+k(w-2),-k]{w-1}}$ counts the exact same set, since the width and $M$ are unchanged, and $|P| = n+k(w-2) + k = m.$

\end{proof}

\printbibliography

\end{document}